\documentclass[oneside,a4paper]{article}
\usepackage{amssymb,amsmath,latexsym,graphicx,tikz,indentfirst,scalerel,hyperref,caption,marginnote,multicol,cite}

\usepackage[english]{babel}
\usepackage{csquotes}

\hypersetup{
  colorlinks   = true, 
  urlcolor     = blue, 
  linkcolor    = blue, 
  citecolor   = red 
}

\usetikzlibrary{automata, positioning, calc,arrows}

\def\dj{d\kern-0.4em\char"16\kern-0.1em}
\def\Dj{\mbox{\raise0.3ex\hbox{-}\kern-0.4em D}}

\newtheorem{theorem}{Theorem}[section]
\newtheorem{lemma}[theorem]{Lemma}
\newtheorem{proposition}[theorem]{Proposition}

\newtheorem{example}[theorem]{Example}

\makeatletter
\renewcommand{\@dotsep}{10000}
\makeatother

\newenvironment{proof}
{\noindent
{\it Proof.}}
{\hspace{\stretch{1}}%
$\Box$}

\newcounter{primer}[section]

\makeatletter
\tikzset{my loop/.style =  {to path={
  \pgfextra{}
  [looseness=4,min distance=5mm]
  \tikz@to@curve@path},font=\sffamily\small
  }}
\makeatletter

\newcommand{\simp}{\stackrel{p}{\sim}}

\newcommand{\sime}{\stackrel{e}{\sim}}

\makeatletter
\newcommand{\myitem}[1]{%
\item[#1]\protected@edef\@currentlabel{#1}%
}
\makeatother

\begin{document}

\thispagestyle{empty}
\begin{center}
\Large{Difference graphs of finite abelian groups with two Sylow subgroups}
\vspace{3mm}

\normalsize{I. Bo\v snjak, R. Madar\' asz, S. Zahirovi\' c}
\end{center}

\begin{abstract}
The power graph and the enhanced power graph of a group $\mathbf G$ are simple graphs with vertex set $G$; two elements of $G$ are adjacent in the power graph if one of them is a power of the other, and they are adjacent in the enhanced power graph if they generate a cyclic subgroup. The difference graph of a group $\mathbf G$, denoted by $\mathcal D(\mathbf G)$, is the difference of the enhanced power graph and the power graph of group $\mathbf G$ with all the isolated vertices removed.  In this paper, we prove that, if a pair of finite abelian groups of order divisible by at most two primes have isomorphic difference graphs, then they are isomorphic.
%
\end{abstract}

\section{Introduction}

In this paper, we deal with graphs defined on groups. In \cite{cameron-graphs-defined-on-groups}, the reader may read about various graphs associated with groups, including the power graph, the enhanced power graph, the deep commuting graph, the commuting graph, and the generating graph. On what could these graphs bring to graph theory and algebra (especially group theory), the reader may refer to \cite{grafovi i grupe}. 
%

The {\bf directed power graph} of a group $\mathbf G$, denoted by $\vec{\mathcal P}(\mathbf G)$, is the simple digraph with vertex set $G$ such that $x\rightarrow y$ if $y$ is a power of $x$. It was introduced by Kelarev and Quinn \cite{kelarev-quinn}. If $y$ is a power of $x$, but $x$ is not a power of $y$, we shall write $x\twoheadrightarrow y$.
%
%
The simple graph obtained by removing the directions of arcs of the directed power graph of a group $\mathbf G$ is called the {\bf power graph} of the group $\mathbf G$; we denote it by $\mathcal P(\mathbf G)$. We write $x\simp y$ if $x$ and $y$ are adjacent in the power graph.
%
%
The {\bf enhanced power graph} of a finite group $\mathbf G$, denoted by $\mathcal P_e(\mathbf G)$, is the simple graph whose vertices are all elements of the group such that two vertices are adjacent if they are both powers of a single element of the group, i.e. if they generate a cyclic subgroup. We write $x\sime y$ if $x$ and $y$ are adjacent in the enhanced power graph. The term enhanced power graph was introduced by Aalipour et al. \cite{on the structure}, although this notion was studied in an even earlier paper, where the authors called it the cyclic graph of a group \cite{abdolahi}. 
%

In \cite{cameron-graphs-defined-on-groups}, Cameron proposed studying differences of these graphs; among others, the difference of the enhanced power graph and the power graph, which is the graph that we call the difference graph of a group. More precisely, the {\bf difference graph} of a group $\mathbf G$ is the graph obtained as the difference of the enhanced power graph and the power graph, and by removing all the isolated vertices. We denote it by $\mathcal D(\mathbf G)$.   Biswas et al. \cite{biswas} studied connectedness and perfectness of the difference graph. They classified all finite nilpotent groups with perfect difference graph. Parveen et al. \cite{parveen-kumar-panda} classified all finite nilpotent groups whose difference graph is split, chordal, dominatable, threshold, or a star graph.

Numerous articles deal with the question when the above-mentioned graphs associated with a group determine the original structure. Also, many authors have studied when these graphs determine each other. Cameron and Ghosh \cite{power graph} proved that every two finite abelian groups with isomorphic power graphs are isomorphic, and they also showed that this does not hold for any two finite groups. Cameron \cite{power graph 2} proved that any two finite groups with isomorphic power graphs have isomorphic directed power graphs. Furthermore, in \cite{nas prvi}, the authors showed that the enhanced power graph and the power graph of any finite group determine each other. Papers by Cameron, Guerra and Jurina \cite{cameron-guerra-jurina} and Zahirovi\' c \cite{zahirovic-1,zahirovic-2} deal with the same question on torsion-free groups, in general. We seek to obtain the analogous result for difference graphs of finite abelian groups. To that end, among other things, we study neighborhood classes and maximal complete bipartite subgraphs of difference graphs of finite abelian groups. As the main result of this paper, we prove that $\mathcal D(\mathbf G)\cong\mathcal D(\mathbf G)$ implies $\mathbf G\cong\mathbf H$ for any pair of finite abelian group $\mathbf G$ and $\mathbf H$ of orders divisible by exactly two primes.
%

\section{Preliminaries}

Here we list some statements that will be useful later in the manuscript.

For graphs $\Gamma$ and $\Delta$, the {\bf strong product} of $\Gamma$ and $\Delta$ is the graph $\Gamma\boxtimes\Delta$ with vertex set $V(\Gamma)\times V(\Delta)$ such that two different vertices $(x_1,y_1)$ and $(x_2,y_2)$ are adjacent in $\Gamma\boxtimes\Delta$ if $x_1$ and $x_2$ are adjacent in $\Gamma$
or $x_1=x_2$, and $y_1$ and $y_2$ are adjacent in $\Delta$ or $y_1=y_2$.

\begin{proposition}[{\cite{nas prvi}}]\label{proizvod}
Let $\mathbf G$ and $\mathbf H$ be finite groups. Then $\mathcal P_e(\mathbf G\times\mathbf H)= \mathcal P_e(\mathbf G)\boxtimes\mathcal P_e(\mathbf H)$  if and only if $\mathrm {gcd}(o(g),o(h))=1$ for all $g\in G$ and $h\in H$.
\end{proposition}

\begin{lemma}[\cite{parveen-kumar-panda}]\label{deljivost}
Let $a$ and $b$ be elements of a finite group $\mathbf G$. If $o(a)$ divides $o(b)$, then $a$ and $b$ are not adjacent in $\mathcal D(\mathbf G)$.
\end{lemma}


\begin{proposition}[\cite{parveen-kumar-panda}]\label{bipartite D}
Let $\mathbf G$ be a finite nilpotent group which is not a $p$-group. Then $\mathcal D(\mathbf G)$ is a bipartite graph if and only if $\mathbf G$ is a direct product of its Sylow-subgroups $\mathbf P_1$ and $\mathbf P_2$, where the exponent of at least one of these subgroups is a prime number.
\end{proposition}

The {\bf order sequence} of a finite group $\mathbf G$ is the list of orders of elements of the group, arranged in non-decreasing order.

The next lemma is a well known statement from group theory. For a proof of the statement and explanation of the connection between the order sequence of a group and some graphs associated to groups, we refer the reader to \cite{niz redova}.

\begin{proposition}\label{niz redova}
Two finite abelian groups have the same order sequences if and only if they are isomorphic.
\end{proposition}

We will also use a following result from group theory. A proof can be seen in \cite{nas prvi}, as a part of the proof of Proposition 5.1.

\begin{lemma}[\cite{niz redova}]\label{ciklicne}
Let $p$ be a prime number and $$\mathbf G=\mathbf C_{p^{r_1}} \times \dots\times \mathbf C_{p^{r_k}}, \; r_1\geq r_2\geq \dots \geq r_k.$$
\begin{enumerate}
\item $\mathbf G$ has exactly $\frac{p^k-1}{p-1}$ subgroups of order $p$.
\item If $\mathbf H$ is a non-maximal cyclic subgroup of the group $\mathbf G$ of order $p^u$, $u>1$, then $H$ is contained in exactly $p^{k-1}$ cyclic subgroups of the group $\mathbf G$ of order $p^{u+1}$.
\end{enumerate}
\end{lemma}

\section{The main results}

It seems natural to ask whether the difference graph $\mathcal D(\mathbf G)$ (with non-empty set of vertices) uniquely determines the finite group $\mathbf G$. As the following example shows, the answer is negative. 

\begin{example}
The dihedral group $\mathbf D_6$, the dyciclic group $\mathbf Q_{12}$, and the cyclic group $\mathbf C_6$ have difference graphs isomorphic to the path of length $2$. Groups $\mathbf D_6$ and $\mathbf Q_{12}$ are non-abelian groups of order 12, and $C_6$ is, of course, an abelian group of order 6.
\end{example}

The question still can be posed for finite abelian groups. Since the vertex set of the difference graph of an abelian $p$-group is empty,   
we will concentrate to the case $\mathbf G=\mathbf P \times \mathbf Q$, where $\mathbf P$ and $\mathbf Q$ are abelian $p$-group and abelian $q$-group, respectively, where $p$ and $q$ are two different primes.

In this section, we will prove that $\mathcal D(\mathbf P \times \mathbf Q)$ uniquely determines the group $\mathbf G=\mathbf P \times \mathbf Q$. In order to do this, we will first study some properties of neighborhood classes of $\mathcal D(\mathbf P \times \mathbf Q)$ and use obtained results to determine some basic parameters of $\mathbf P$ and $\mathbf Q$. Then we will make use of some maximal complete bipartite subgraphs of $\mathcal D(\mathbf P \times \mathbf Q)$ to prove the main result of the paper.

In the further text, let $\mathbf P=\mathbf C_{p^{r_1}} \times \dots \times\mathbf C_{p^{r_k}}$ and $\mathbf Q=\mathbf C_{q^{t_1}} \times \dots \times\mathbf C_{q^{t_l}}$, where $r_1\geq r_2\geq \dots \geq r_k$ and $t_1\geq t_2\geq \dots \geq t_l$. Also, put $r_1+ \dots +r_k=n$ and $t_1+ \dots +t_l=m$.

First we will resolve a question when two vertices are adjacent in $\mathcal D(\mathbf G)$.

\begin{lemma}\label{susednost}
Let $a_1, a_2 \in P$ and $b_1, b_2 \in Q$. Then $(a_1,b_1)$ and $(a_2,b_2)$ are adjacent in $\mathcal D(\mathbf G)$ if and only if $a_1\twoheadrightarrow a_2$ and $b_2\twoheadrightarrow b_1$ or $a_2\twoheadrightarrow a_1$ and $b_1\twoheadrightarrow b_2$.
\end{lemma}
\begin{proof}
If $(a_1,b_1)$ and $(a_2,b_2)$ are adjacent in $\mathcal D(\mathbf G)$, then they are adjacent in $\mathcal P_e(\mathbf G)$, and according to Proposition \ref{proizvod}, $a_1$ and $a_2$ generate a cyclic group, as well as $b_1$ and $b_2$. According to Lemma \ref{deljivost}, $o(a_1,b_1)=o(a_1)\cdot o(b_1)$ and $o(a_2,b_2)=o(a_2)\cdot o(b_2)$ do not divide each other. Therefore, $a_1$ and $a_2$, and $b_1$ and $b_2$ also, cannot be of the same order. Moreover, if $o(a_1)<o(a_2)$, than $o(b_1)>o(b_2)$, and conversely, if $o(a_1)>o(a_2)$, then $o(b_1)<o(b_2)$. In a cyclic group of a prime power order, $o(x)>o(y)$ is equivalent to $x\twoheadrightarrow y$, so the proof has been completed.
\end{proof}\\

By $M(\mathbf K)$ we will denote the set of all generators of maximal cyclic subgroups of a finite group $\mathbf K$.

\begin{lemma}\label{cvorovi diferenc grafa}
Let $a \in P$ and $b \in Q$. Then $(a,b)$ is not a vertex of $\mathcal D(\mathbf G)$ if and only if $a=b=e$ or $a\in M(\mathbf P)$ and $b\in M(\mathbf Q)$.
\end{lemma}
\begin{proof}
A direct consequence of a Lemma \ref{susednost}.
\end{proof}\\

By $N(x)$ we will denote the {\bf neighborhood} of a vertex $x$ (the set of all neighbors of $x$ in a given graph). The {\bf neighborhood class} of a vertex $x$ is the set of all vertices who have the same neighborhood as $x$. We will denote it by $[x]$.

Now we wish to describe neighborhood classes in difference graphs of a group $\mathbf P\times\mathbf Q$, where $\mathbf P$ and $\mathbf Q$ are defined above. 

\begin{lemma}\label{srednje klase}
Let $a \in P$ and $b \in Q$, where $a,b \neq e$, $a\notin M(\mathbf P)$, $b\notin M(\mathbf Q)$.
Then $[(a,b)]=\{(x,y) :  \langle a\rangle = \langle x\rangle, \; \langle b\rangle = \langle y\rangle \}$.
\end{lemma}
\begin{proof}
Let $S=\{(x,y) :  \langle a\rangle = \langle x\rangle, \; \langle b\rangle = \langle y\rangle \}$. It is clear that all vertices from $S$ have the same neighborhood as $(a,b)$. Namely, if $(x,y)\in S$, then $\{z\in Y : \; y\twoheadrightarrow z\}=\{z\in Y : \; b\twoheadrightarrow z\}$ and $\{z\in Y : \; z\twoheadrightarrow y\}=\{z\in Y : \; z\twoheadrightarrow b\}$, and the same holds for $a$ and $x$.\\
Let $N((a,b))=N((x,y))$. We will consider three cases.\\
1) Let $x$ and $a$ do not belong to the same cyclic subgroup of $\mathbf P$. Let $z$ be a generator of a maximal cyclic subgroup that contains $a$. Then $(z,e)\in N((a,b))$ and $(z,e) \notin N((x,y))$.\\
2) Let $a\twoheadrightarrow x$. Let $t$ be a generator of a maximal cyclic subgroup of $\mathbf Q$ that contains $b$.
Then $(x,t)\in N((a,b))$ and $(x,t) \notin N((x,y))$.\\
3) Let $x\twoheadrightarrow a$. Then $(x,e)\in N((a,b))$ and $(x,e) \notin N((x,y))$.\\
The only remaining possibility is that $\langle a\rangle = \langle x\rangle$. Analogously, $\langle b\rangle$ must be equal to $\langle y\rangle$.
\end{proof}\\

This type of neighborhood classes of a difference graph we will call $m${\bf -classes}. If $o(a)=p^u$ and $o(b)=q^v$, then $|[(a,b)]|=(p^u-p^{u-1})(q^v-q^{v-1})$, because $p^u-p^{u-1}$ is the number of generators of the group $\langle a\rangle$, and $q^v-q^{v-1}$ is the number of generators of the group $\langle b\rangle$.

\begin{lemma}\label{e klase}
Let $b\in Q$, $b\neq e$, $o(b)=q^u$. Then
$[(e,b)]=\{(e,y) : o(y)=q^u \; \mbox{i} \; | \langle b\rangle \cap \langle y\rangle |\geq q^{u-1} \}$.
\end{lemma}
\begin{proof}
Let $S=\{(e,y) : o(y)=q^u \; \mbox{i} \; | \langle b\rangle \cap \langle y\rangle |\geq q^{u-1} \}$. It is clear that all vertices from $S$ have the same neighborhood as $(e,b)$. Namely, if $(e,y)\in S$, then $\{z\in Q : \; y\twoheadrightarrow z\}=\{z\in Q : \; b\twoheadrightarrow z\}$.
\\Let $N((e,b))=N((x,y))$ and $x\neq e$. Then $(x,e)\in N((e,b))$ and $(x,e) \notin N((x,y))$. So, it must be $x=e$.
Now, let $N((e,b))=N((e,y))$. Then $\{z\in Q : \; y\twoheadrightarrow z\}=\{z\in Q : \; b\twoheadrightarrow z\}$, which means that $b$ and $y$ have the same cyclic subgroups. Therefore, $b$ and $y$ have the same order, and either $\langle b\rangle = \langle y\rangle$, or $\langle b\rangle$ i $\langle y\rangle$ have a common cyclic subgroup of order $q^{u-1}$.
This means that $o(y)=q^u$ i $| \langle b\rangle \cap \langle y\rangle |\geq q^{u-1}$.
\end{proof}\\

This type of neighborhood classes of a difference graph we will call $l${\bf -classes}. If $o(b)=q$, then there exist $q^{l}-1$ elemenats $y\in Q$ of order $q$, and for all of them it holds $N((e,b))=N((e,y))$. This means that $|[(e,b)]|=q^{l}-1$.
If $o(b)=q^u$, $u\geq 2$, then, according to Lemma \ref{ciklicne}, there exist $q^{l-1}$ cyclic subgroups of the group $Q$ which have the intersection of cardinality $q^{u-1}$ or $q^u$ with $\langle b\rangle$. This implies $|[(e,b)]|=q^{l-1}(q^u-q^{u-1})$.

\begin{lemma}\label{maks klase}
Let $a \in P$ and $b \in Q$, where $a,b \neq e$, $a\notin M(\mathbf P)$,  $o(b)=q^u$ and $b\in M(\mathbf Q)$.
Then $[(a,b)]=\{(x,y) :  \langle a\rangle = \langle x\rangle, \; o(y)=q^u, \; | \langle b\rangle \cap \langle y\rangle |\geq q^{u-1} \; \mbox{i} \; y\in M(\mathbf Q) \}$.
\end{lemma}
\begin{proof}
Let
$S=\{(x,y) :  \langle a\rangle = \langle x\rangle, \; o(y)=q^u, \; | \langle b\rangle \cap \langle y\rangle |\geq q^{u-1} \; \mbox{i} \; y\in M(\mathbf Q) \}$. It is clear that all vertices from $S$ have the same neighborhood as $(a,b)$. Namely, if $(x,y)\in S$, then $\{z\in Y : \; y\twoheadrightarrow z\}=\{z\in Y : \; b\twoheadrightarrow z\}$ i \\$\{z\in Y : \; z\twoheadrightarrow y\}=\{z\in Y : \; z\twoheadrightarrow b\}=\emptyset$.\\
Let $N((x,y))=N((a,b))$. According to Lemmas \ref{srednje klase} i \ref{e klase}, exactly one of the elements $x$ i $y$ generates a maximal cyclic subgroup in its group, and the other can not be equal to $e$. Let $x$ generates a maximal cyclic subgroup in $P$ and let $t$ be a generator of a maximal cyclic subgroup that contains $y$. Then $(e,t)\in N((x,y))$ and $(e,t) \notin N((a,b))$. Therefore, $y$ must generate a maximal cyclic subgroup in $Q$. But then $\langle b\rangle$ and $\langle y \rangle$ must contain the same cyclic subgroups, and $a$ and $x$ must be contained in the same cyclic subgroups. As in a proof of the previous  Proposition, this means that $o(y)=q^u$ and $| \langle b\rangle \cap \langle y\rangle |\geq q^{u-1}$, and that $a$ and $x$ generate the same cyclic subgroup.
\end{proof}\\

This type of neighborhood classes of a difference graph we will call $u${\bf -classes}. Cardinality of these classes does not only depend on orders of elements $a$ and $b$. So we will consider this question later.

The following statement deals with a special case which we want to exclude from further consideration.

\begin{lemma}\label{spec slucaj}
Let $\mathbf G=\mathbf P\times \mathbf Q$, where $\mathbf P=(\mathbf C_{p})^k$ and
$\mathbf Q=\mathbf C_{q^{t_1}} \times \dots \times\mathbf C_{q^{t_l}}$ and let $\mathbf H$ be an abelian group such that $\mathcal D(\mathbf G)\cong \mathcal D(\mathbf H)$. Then $\mathbf G\cong \mathbf H$.

\end{lemma}
\begin{proof}
According to Proposition \ref{bipartite D}, the graph $\mathcal D(\mathbf G)$ is bipartite and every abelian group whose difference graph is isomorphic to $\mathcal D(\mathbf G)$  is the direct product of an abelian group of a prime exponent and an abelian group of prime power order. We need to prove that these two groups must be isomorphic to $\mathbf P$ and $\mathbf Q$ respectively. One partite set of $\mathcal D(\mathbf G)$ (denoted by $K_e$) consists of elements $(e,b)$, where $b\in Q\setminus\{e\}$. Another partite set (denoted by $K_m$) consists of elements $(a,b)$, where $a\in P\setminus \{e\}$, $b\in Q$, and $b$ does not generate a maximal cyclic subgroup. 

Suppose that $q\neq 2$. The smallest neighborhood class in $K_e$ consists of elements $(e,b)$, where $o(b)=q$. This neighborhood class has $q^{l}-1$ vertices, while remaining classes have $q^{l-1}(q^u-q^{u-1})$ vertices, $u\geq 2$. The smallest neighborhood class in $K_m$ consists of elements $(a,e)$, $a\neq e$. There are $p^k-1$ of them, while other neighborhood classes of $K_m$ have $(p^k-1)(q^u-q^{u-1})$ vertices, $u\geq 1$. If the difference graph is complete bipartite, then $\mathbf Q=(\mathbf C_{q})^l$, and it is an easy task to reconstruct $\mathbf P$ and $\mathbf Q$. If not, only one of two smallest neighborhood classes  has the property that its vertices are adjacent to all vertices from the other partite set. This neighborhood class is the smallest class of $K_m$, and it has $p^k-1$ vertices. This means that $p$ and $k$ are uniquely determined. Now we can observe the partite set $K_e$, and from its neighborhood classes we can determine order sequence of the group $\mathbf Q$. Namely, the number of vertices in the neighborhood class of cardinality $q^{l}-1$ is equal to the number of elements of order $q$, and the total number of vertices in neighborhood classes of cardinality $q^{l-1}(q^u-q^{u-1})$ is equal to the number of elements of order $q^u$. According to Proposition \ref{niz redova}, by determining the order sequence, we have uniquely determined the group $\mathbf Q$.

For $q=2$, there will be more then one neighborhood class of cardinality $p^k-1$ in the set $K_m$, but only one of them has all vertices from the other partite set as neighbors, so we can proceed as in the case $q\neq 2$.
\end{proof}\\

For the rest of the paper, we will work under the assumption that neither of two Sylow subgroups $\mathbf P$ and $\mathbf Q$ has a prime exponent.

Let $a$ generate a biggest maximal cyclic subgroup in $\mathbf P$. Let $b$ generates a maximal proper cyclic subgroup of a biggest maximal cyclic subgroup  in $\mathbf Q$. The element $a$ we will call a {\bf max element} in $\mathbf P$, and the element $b$ a {\bf submax element} in $\mathbf Q$. It is clear that $o(a)=p^{r_1}$ and $o(b)=q^{t_1-1}$. The neighborhood class of $(a,b)$ we will call {\bf max-submax $PQ$-class}, or shortly, only a $PQ$-class. Max-submax $QP$-class is defined in an analogous way. These classes are special $u$-classes. All $PQ$-classes (and conversely all $QP$-classes) have the same cardinality. Let us observe the neighborhood class of $(a,b)$. According to Lemma \ref{ciklicne}, having in mind that all cyclic subgroups of the same order as $\langle a\rangle$ are maximal, we have that there exists exactly $p^{k-1}$ maximal cyclic subgroups of $\mathbf P$ which have the same maximal proper subgroup as $\langle a\rangle$. This implies $|[(a,b)]|=p^{k-1}(p^{r_1}-p^{r_1-1})(q^{t_1-1}-q^{t_1-2})$. Analogously, the cardinality of $QP$-classes is $q^{l-1}(q^{t_1}-q^{t_1-1})(p^{r_1-1}-p^{r_1-2})$.

\begin{lemma}\label{najvece klase}
The biggest neighborhood classes in $\mathcal D(\mathbf P \times \mathbf Q)$ are max-submax $PQ$-classes or max-submax $QP$-classes.
\end{lemma}
\begin{proof}
The cardinality of the neighborhood class of element $(e,b)$ can not exceed $q^{l-1}(q^{t_1}-q^{t_1-1})$. This is almost always less than the cardinality of a $QP$-class. Only for $p=2$, $r_1=2$, if $b$ is a max element in $\mathbf Q$, $[(e,b)]$ and every $QP$-class will be of same size. In this case, vertices of the class $[(e,b)]$ have more neighbors than vertices of a $QP$-class and we can distinguish them that way. $m$-classes are obviously smaller than max-submax classes. The cardinality of a $u$-class $[(a,b)]$, where $a$ generates a maximal cyclic subgroup in $P$, is $|[(a,b)]|=z(p^{u}-p^{u-1})(q^{v}-q^{v-1})$, where $1\leq z \leq p^{k-1}$, $u\leq r_1$, $v\leq t_1-1$. If the class in question is not a max-submax class, it is smaller than a $PQ$-class. This means that the biggest neighborhood classes are $PQ$-classes or $QP$-classes, except in a special case when some other classes could be of the same cardinality, but we have already explained how to distinguish them from max-submax classes.
\end{proof}

\begin{theorem}\label{parametri}
Let $\mathbf P=\mathbf C_{p^{r_1}} \times \dots \times\mathbf C_{p^{r_k}}$ and $\mathbf Q=\mathbf C_{q^{t_1}} \times \dots \times\mathbf C_{q^{t_l}}$, where $r_1\geq r_2\geq \dots \geq r_k$ and $t_1\geq t_2\geq \dots \geq t_l$, and let $n=r_1+ \dots +r_k$ and $m=t_1+ \dots +t_l$.
The difference graph $\mathcal D(\mathbf P\times \mathbf Q)$ uniquely determines parameters $p,q,n,m,k,l,r_1,t_1$.
\end{theorem}

\begin{proof}
According to Lemma \ref{najvece klase}, the biggest neighborhood classes are max-submax classes. Without loss of generality, we may assume that $PQ$-classes are biggest. Take a look at the neighborhood classes that contain the neighbors of elements of a $PQ$-class (let it be the class $[(a,b)]$). These classes are of the following cardinalities: $q^{l-1}(q^{t_1}-q^{t_1-1})(p^{r_1-1}-p^{r_1-2})$, $q^{l-1}(q^{t_1}-q^{t_1-1})(p^{r_1-2}-p^{r_1-3})$, $\dots$, $q^{l-1}(q^{t_1}-q^{t_1-1})(p-1)$, $q^{l-1}(q^{t_1}-q^{t_1-1})$. From there we can determine $p$ and $r_1$. The biggest class which contains neighbors of the class $[(a,b)]$ is a $QP$-class. Similar to the case of $PQ$-classes, we can use neighbors of the observed $QP$-class to determine $q$ and $t_1$. Knowing those four parameters and cardinalities of above mentioned classes of neighbors, we can now determine $k$ and $l$. Finally, let us note that the smallest class containing the neighbors of the class $[(a,b)]$ is an $l$-class, and its elements have precisely $(p^n-1)q^{t_1-1}$ neighbors. From this we will determine $n$, and $m$ can be determined in a similar way.
\end{proof}\\

In order to make a complete description of the groups $\mathbf P$ and $ \mathbf Q$, we will consider maximal complete bipartite subgraphs of $\mathcal D(\mathbf P \times \mathbf Q)$, and neighborhood classes of which these subgraphs consist. First, let us note that one of  maximal complete bipartite subgraphs is induced by the set of vertices $X_e \cup Y_e$, where $X_e=\{(e,b) \; : \; b\in Q\setminus \{e\}\}$ and $Y_e=\{(a,e) \; : \; a\in P\setminus \{e\}\}$. Moreover, $X_e$ and $Y_e$ are partite sets of this subgraph. We will denote this subgraph by $H_e$. It is obviously complete bipartite, and it is maximal for, if $a,b \neq e$, then $(a,b)$ is not a neighbor of either $(a,e)$ or $(e,b)$. Neighborhood classes of $H_e$ are  $l$-classes of $\mathcal D(\mathbf P \times \mathbf Q)$. The smallest class in $X_e$ is $[(e,b)]$, where $o(b)=q$, and it has $q^l-1$ vertices. Analogously, the smallest class in $Y_e$ is $[(a,e)]$, where $o(a)=p$, and it has $p^k-1$ vertices. Above mentioned smallest classes have the property that their neighborhoods in $\mathcal D(\mathbf G)$ are equal to the other partite set.
For such classes (or vertices) in a complete bipartite subgraph we will say that they have the SDK-property.

Now we will pay attention to other maximal bipartite subgraphs.

\begin{lemma}\label{acdb} Let $H$ be a complete bipartite subgraph of $\mathcal D(\mathbf G)$ with partite sets $X$ and $Y$. Further, let $(a_1,b_1)\in X$, $(c_1,d_1)\in Y$, and $c_1\twoheadrightarrow a_1$, $b_1\twoheadrightarrow d_1$. Then for all $(a,b)\in X$, $(c,d)\in Y$ it holds: $c\twoheadrightarrow a$ and $b\twoheadrightarrow d$.

\end{lemma}
\begin{proof}
Let $(a,b)\in X$, $(c,d)\in Y$. If $a_1\twoheadrightarrow c$, then it must be $d\twoheadrightarrow b_1$, since $(a_1,b_1)$ and $(c,d)$ are neighbors in $\mathcal D(\mathbf G)$. However, then $c_1\twoheadrightarrow c$ and $d\twoheadrightarrow d_1$, which means that $(c,d)$ and $(c_1,d_1)$ are neighbors in the difference graph, which is impossible. So, it must be $c\twoheadrightarrow a_1$ and $b_1\twoheadrightarrow d$. Suppose now that $a\twoheadrightarrow c$ and $d\twoheadrightarrow b$. Then $a\twoheadrightarrow a_1$ and $b_1\twoheadrightarrow b$ must hold, and then $(a,b)$ and $(a_1,b_1)$ would be neighbors in $\mathcal D(\mathbf G)$, and this is obviously not true. Therefore, it must be $c\twoheadrightarrow a$ and $b\twoheadrightarrow d$.
\end{proof}

\begin{lemma}\label{AD} Let $H$ be a maximal complete bipartite subgraph of the graph $\mathcal D(\mathbf G)$, with partite sets $X$ and $Y$. Also, let for all $(a,b)\in X$, $(c,d)\in Y$ it holds: $c\twoheadrightarrow a$ and $b\twoheadrightarrow d$. If $A=\{a\in P \; : \; (\exists b\in Q) \; (a,b) \in X \}$ then  $A$ is a non-maximal cyclic subgroup of $\mathbf P$.

\end{lemma}
\begin{proof}
Let $a', a''\in A$. For any $(c,d)\in Y$, $a'$ and $a''$ belong to  $\langle c\rangle$. Since $\mathbf P$ is a $p$-group, it must be $a'\simp a''$. If $a$ is an element of a maximal order from $A$, then $A\subseteq  \langle a \rangle$. Analogously, if $D=\{d\in Q \; : \; (\exists c\in P) \; (c,d) \in Y \}$, then there exists $d\in D$ such that $D\subseteq  \langle d \rangle$. We will show that $A= \langle a \rangle$.

If a vertex from a neighborhood class belongs to the graph $H$, then, because of the maximality of $H$, the whole neighborhood class belongs to it. Consequently, if a generator of a cyclic subgroup of the group $\langle a \rangle$ belongs to $A$, then $A$ contains all generators of that cyclic subgroup. 

Suppose that $e$ does not belong to $A$. Let $\langle b' \rangle$ be a minimal cyclic subgroup of $\mathbf Q$ which contains $\langle d \rangle$ as a proper subgroup. The subgraph of $\mathcal D(\mathbf G)$ induced by $H \cup \{(e,b')\}$ is a complete bipartite subgraph with partite sets $X \cup \{(e,b')\}$ and $Y$. Since $H$ is a maximal complete bipartite subgraph, $A$ must contain $e$. 

Suppose now that $a'\in \langle a \rangle$ and that no generator of the subgroup $\langle a' \rangle$ belongs to $A$. Let $B'=\{b\in Q \; : \; (\exists x\in \langle a' \rangle) \; (x,b) \in X \}$ and $b''$ be an element of a maximal order from $B'$. Then no element $b\in B'$ can satisfy $b\twoheadrightarrow b''$. The subgraph of $\mathcal D(\mathbf G)$ induced by $H \cup \{(a',b'')\}$ is a complete bipartite subgraph with partite sets $X\cup \{(a',b'')\}$ and $Y$. Namely, $(a',b'')$ is obviously adjacent to all vertices from $Y$ and obviously non-adjacent to all $(a'',b)$ from $X$ such that $a''\in \langle a' \rangle$. If $(a',b'')$ is adjacent to $(a''',b)\in X$ such that $a'''\twoheadrightarrow a'$, then $(a''',b)$ is adjacent to $(a'',b'')$ from $X$, where $a''$ is some element from $\langle a' \rangle$, which is impossible. Therefore, all generators of cyclic subgroups of $\langle a \rangle$ are in $A$, which implies $A= \langle a \rangle$.
\end{proof}

\begin{lemma}\label{SDK}
Let $H$ be a maximal complete bipartite subgraph of the graph $\mathcal D(\mathbf G)$, with partite sets $X$ and $Y$. If $X$ and $Y$ contain elements (classes) with the SDK-property, then $H=H_e$, or $A\neq \{e\}$, $D\neq \{e\}$, and all cyclic subgroups of $\mathbf P$ or $\mathbf Q$ containing $A$ or $D$ as proper subgroups are maximal.
\end{lemma}

\begin{proof}
If $a'\in A=\langle a \rangle$ and $a'$ does not generate $A$, then $(a,'b)\in X$ can not have the SDK-property, since it is adjacent to $(a,e)\notin Y$.
If $D\neq \{e\}$ and there exist $c',c''\in P\setminus A$ such that $c'\twoheadrightarrow c''$ and $c''\twoheadrightarrow a$, then $(a,b)\in X$ can not have the SDK-property, since it has two neighbors, $(c',e)$ and $(c'',d)$, which are adjacent to each other. Therefore, there are three possibilities.

If $A=\{e\}$, $D=\{e\}$, then $H=H_e$.

If $A= \{e\}$, $D\neq \{e\}$, then $\mathbf P=(\mathbf C_p)^k$, and this case has been already studied in Lemma \ref{spec slucaj}.

If $A\neq \{e\}$, $D\neq \{e\}$, then all cyclic subgroups which contain $A$ or $D$ are maximal.
\end{proof}

\begin{theorem}\label{glavna}
Let $\mathbf G$ and $\mathbf K$ be finite abelian groups having exactly two Sylow subgroups. If $\mathcal D(\mathbf G)$ and $\mathcal D(\mathbf K)$ are isomorphic, then $\mathbf G$ and $\mathbf K$ are isomorphic groups.
\end{theorem}
\begin{proof}
Let $\mathbf G=\mathbf P\times \mathbf Q$, where $\mathbf P$ and $\mathbf Q$ are as defined in the statement of Theorem \ref{parametri}. According to Theorem \ref{parametri}, we can determine parameters $p,q,n,m,k,l,r_1,t_1$. Now, we will try to distinguish the subgraph $H_e$ from other maximal complete bipartite subgraphs of the graph $\mathcal D(\mathbf G)$. Let us note that exactly one neighborhood class from both $X_e$ and $Y_e$ have the SDK-property. Cardinalities of classes in question are $q^l-1$ and $p^k-1$. According to Lemma \ref{SDK}, in every other maximal complete bipartite subgraph which have two classes with the SDK-property, cardinalities of these classes are $(p^u-p^{u-1})q^{l-1}(q^{v+1}-q^v)$ and $(q^v-q^{v-1})p^{k-1}(p^{u+1}-p^u)$, which is greater than cardinalities of corresponding classes from $H_e$. Thus we have determined $H_e$. Now we can determine the numbers of elements of any order in groups $\mathbf P$ and $\mathbf Q$. Namely, there are exactly $q^l-1$ elements of order $q$ in $\mathbf Q$, while the total number of elements in neighborhood classes of $X$ of cardinality $q^{l-1}(q^u-q^{u-1})$ is equal to the number of elements of order $q^u$ in $\mathbf Q$. Therefore, in this way we can determine order sequences of groups $\mathbf P$ and $\mathbf Q$, and then, according to Proposition \ref{niz redova}, the groups $\mathbf P$ and $\mathbf Q$ are uniquely determined.
\end{proof}

%

\vbox{
\vbox{\noindent
Ivica Bo\v snjak

University of Novi Sad, Department of Mathematics and Informatics, Serbia

{\it e-mail}: \href{mailto:ivb@dmi.uns.ac.rs}{ivb@dmi.uns.ac.rs}}\

\vbox{\noindent
Roz\' alia Madar\' asz

University of Novi Sad, Department of Mathematics and Informatics, Serbia

{\it e-mail}: \href{mailto:rozi@dmi.uns.ac.rs}{rozi@dmi.uns.ac.rs}}\

\vbox{\noindent
Samir Zahirovi\' c

University of Novi Sad, Department of Mathematics and Informatics, Serbia

{\it e-mail}: \href{mailto:samir.zahirovic@dmi.uns.ac.rs}{samir.zahirovic@dmi.uns.ac.rs}}}


\begin{thebibliography}{99}\addcontentsline{toc}{chapter}{References}
\bibitem{on the structure} G. Aalipour, S. Akbari, P.J. Cameron, R. Nikandish, F. Shaveisi, {\it On the structure of the power graph and the enhanced power graph of a group}, Electron. J. Combin. {\bf 24} (2017), no. 3, Paper 3.16, 18 pp.
\bibitem{abdolahi} A. Abdollahi, A. Mohammadi Hassanabadi, {\it Noncyclic Graph of a Group}, Communications in Algebra {\bf 35} (2007), no. 7, 2057-2081.
\bibitem{biswas} S. Biswas, P.J. Cameron, A. Das, H.K. Dey, {\it On the difference of the enhanced power graph and the power graph of a finite group}, \href{https://arxiv.org/abs/2206.12422}{arXiv:2206.12422}
\bibitem{cameron-graphs-defined-on-groups} P.J. Cameron, {\it Graphs defined on groups}, International Journal of Group Theory (2021), doi: \href{http://dx.doi.org/110.22108/ijgt.2021.127679.1681}{110.22108/ijgt.2021.127679.1681}
\bibitem{grafovi i grupe} P.J. Cameron, {\it What can graphs and algebraic structures say to each other?}, AKCE International Journal of Graphs and Combinatorics doi: \href{https://doi.org/10.1080/09728600.2023.2290036}{10.1080/09728600.2023.2290036}
\bibitem{power graph 2} P.J. Cameron, {\it The power graph of a finite group, II}, J. Group Theory {\bf 13} (2010), no. 6, 779-783.
\bibitem{power graph} P.J. Cameron, S. Ghosh, {\it The power graph of a finite group}, Discrete Math. {\bf 311} (2011), no. 13, 1220-1222.
\bibitem{niz redova} P.J. Cameron, H.K. Dey, {\it On the order sequence of a group}, \href{https://arxiv.org/abs/2310.06516}{arXiv:2310.06516v1}
\bibitem{cameron-guerra-jurina} P.J. Cameron, H. Guerra, \v S. Jurina, {\it The power graph of a torsion-free group}, Journal of Algebraic Combinatorics {\bf 49} (2019), no. 1, 83-98.
\bibitem{kelarev-quinn} A. V. Kelarev, S. J. Quinn, {\it A combinatorial property and power graphs of groups}, Contributions to general algebra {\bf 12} (2000) 229-235.


\bibitem{parveen-kumar-panda} Parveen, J. Kumar, R.P. Panda, {\it On the Difference graph of power graphs of a finite groups}, Quaestiones Mathematicae 2023, doi: \href{https://doi.org/10.2989/16073606.2023.2278078}{10.2989/16073606.2023.2278078.}
\bibitem{zahirovic-1} S. Zahirovi\' c, {\it The power graph of a torsion-free group of nilpotency class $2$}, Journal of Algebraic Combinatorics {\bf 55} (2022), no. 3, 715-727.
\bibitem{zahirovic-2} S. Zahirovi\' c, {\it The power graph of a torsion-free group determines the directed power graph}, Discrete Applied Mathematics {\bf 305} (2021), 109-118.
\bibitem{nas prvi} S. Zahirovi\' c, I. Bo\v snjak, R. Madar\' asz,  {\it A Study of Enhanced Power Graphs of Finite Groups}, Journal of Algebra and Its Applications {\bf 19} (2019), no. 4, 2050062
\end{thebibliography}
\end{document}